\documentclass[12pt,a4paper]{article}
\usepackage[dvips]{graphicx,color} 
\usepackage[T1]{fontenc}  
\usepackage{times}     
\usepackage{enumerate,amsmath,amsthm,mathrsfs,citesort}
\usepackage{amsfonts}
\usepackage{amssymb}
\usepackage[a4paper,left=2.5cm,right=2cm,top=2.5cm,bottom=2.5cm]{geometry}
\usepackage{amstext} 
\usepackage{array}   
\usepackage[textsize=scriptsize,textwidth=4cm,linecolor=cyan, backgroundcolor=cyan]{todonotes}
\newtheoremstyle{bthm}{\baselineskip}{\baselineskip}{\slshape}{}{\bfseries}{}{  }{}
\newtheoremstyle{bex}{\baselineskip}{\baselineskip}{}{}{\itshape}{}{  }{}
\theoremstyle{bthm}
\newtheorem{theorem}{Theorem}[section]

\newtheorem{lemma}[theorem]{Lemma}
\newtheorem{proposition}[theorem]{Proposition}

\newtheorem{fact}[theorem]{Fact}

\newtheorem{conjecture}[theorem]{Conjecture}
\theoremstyle{bex}

\begin{document}
\begin{titlepage}
\title{ Linear $\chi$-binding functions for $\{P_3\cup P_2, gem\}$-free graphs}
\author{Athmakoori Prashant$^{1,}$\footnote{The author's research was supported by the Council of Scientific and Industrial Research,  Government of India, File No: 09/559(0133)/2019-EMR-I.}, S. Francis Raj$^{2}$ and M. Gokulnath$^{3}$}
\date{{\footnotesize $^{1,2}$ Department of Mathematics, Pondicherry University, Puducherry-605014, India.\\
$^{3}$Department of Mathematics, School of Advanced Sciences, VIT-AP University, Amaravati-522241, \\
\vskip-.2cm
Andhra Pradesh, India.}\\
{\footnotesize$^{1}$: 11994prashant@gmail.com\, $^{2}$: francisraj\_s@pondiuni.ac.in\, $^{3}$: gokulnath@vitap.ac.in\ }}
\maketitle
\renewcommand{\baselinestretch}{1.3}\normalsize
\begin{abstract}
Finding families that admit a linear $\chi$-binding function is a problem that has interested researchers for a long time. Recently, the question of finding linear subfamilies of $2K_2$-free graphs has garnered much attention. In this paper, we are interested in finding a linear subfamily of a specific superclass of $2K_2$-free graphs, namely $(P_3\cup P_2)$-free graphs. We show that the class of $\{P_3\cup P_2,gem\}$-free graphs admits $f=2\omega$ as a linear $\chi$-binding function. Furthermore, we give examples to show that the optimal $\chi$-binding function $f^*\geq \left\lceil\frac{5\omega(G)}{4}\right\rceil$ for the class of $\{P_3\cup P_2, gem\}$-free graphs and that the $\chi$-binding function $f=2\omega$ is tight when $\omega=2\textnormal{ and } 3$.
\end{abstract}
\section{Introduction}\label{intro}
All graphs considered in this paper are simple, finite and undirected.
Let $G$ be a graph with vertex set $V(G)$ and edge set $E(G)$.
Given a positive integer $n$, let $P_n,C_n$ and $ K_n$ respectively denote the path, the cycle and the complete graph on $n$ vertices.
If $G_1$ and $G_2$ are two vertex-disjoint graphs, then their \emph{union} $G_1\cup G_2$ is the graph with vertex set $V(G_1\cup G_2)=V(G_1)\cup V(G_2)$ and edge set $E(G_1\cup G_2)=E(G_1)\cup E(G_2)$.
Similarly, the \emph{join} of $G_1$ and $G_2$, denoted by $G_1+G_2$, is the graph whose vertex set $V(G_1+G_2) = V(G_1)\cup V(G_2)$ and the edge set $E(G_1+G_2) = E(G_1)\cup E(G_2)\cup\{xy: x\in V(G_1),\ y\in V(G_2)\}$. 
Let $\omega(G)$ and $\alpha(G)$ denote the \emph{clique number} and  \emph{independence number} of a graph $G$ respectively.
When there is no ambiguity, $\omega(G)$ will be denoted by $\omega$.
For $S,T\subseteq V(G)$, let $N_T(S)=N(S)\cap T$ (where $N(S)$ denotes the set of all neighbors of $S$ in $G$),
let $\langle S\rangle$ denote the subgraph induced by $S$ in $G$ and let $[S,T]$ denote the set of all edges with one end in $S$ and the other end in $T$.
If each vertex in $S$ is adjacent with every vertex in $T$, then $ [S, T ]$ is said to be complete.
For any graph $G$, let $\overline{G}$ denote the complement of $G$.
If $H$ is an induced subgraph of $G$, we shall denote it by $H\sqsubseteq G$.

For any positive integer $k$, a \emph{proper $k$-coloring} of a graph $G$ is a mapping $c$ : $V(G)\rightarrow\{1,2,\ldots,k\}$ such that $c(u)\neq c(v)$ whenever $uv\in E(G)$.
We say that $G$ is \emph{$k$-colorable} if it admits a proper $k$-coloring.
The \emph{chromatic number} of $G$, denoted by $\chi(G)=\min\{k: G \textnormal{ is } \linebreak k\textnormal{-colorable}\}$.

Let $\mathcal{F}$ be a family of graphs.
We say that $G$ is \emph{$\mathcal{F}$-free} if it does not contains any induced subgraph which is isomorphic to a graph in $\mathcal{F}$.
For a fixed graph $H$, let us denote the family of $H$-free graphs by $\mathcal{G}(H)$.
We say that a graph $G$ is \emph{perfect} if $\chi(H)=\omega(H)$ for every $H\sqsubseteq G$.
%
%
As a generalization of the concept of perfect graphs, A. Gy{\'a}rf{\'a}s in \cite{gyarfas1987problems}, introduced the concept of $\chi$-binding functions as follows.
A hereditary graph class $\mathcal{G}$ of graphs is said to be \emph{$\chi$-bounded} if there is a function $f$ (called a $\chi$-binding function) such that $\chi(G)\leq f(\omega(G))$, for every $G\in \mathcal{G}$.          
Let $f^*$ denote the \emph{optimal} $\chi$-binding function.       
We say that the $\chi$-binding function $f$ is \emph{special linear} if $f(x)= x+c$, where $c$ is a constant.
There has been extensive research done on $\chi$-binding functions for various graph classes.
See for instance, the survey articles in \cite{scott2020survey,schiermeyer2019polynomial}.
In \cite{gyarfas1987problems}, A. Gy{\'a}rf{\'a}s posed a lot of open problems and conjectures which still remain open.
The following was an interesting conjecture which was given by A. Gy{\'a}rf{\'a}s.

\begin{conjecture}[\cite{gyarfas1987problems}]
$\mathcal{G}(H)$ is $\chi$-bound for every fixed forest $H$.
\end{conjecture}
One of the earlier works in this direction was done by S. Wagon in \cite{wagon1980bound}, where he showed that the class of $2K_2$-free graphs admits $\binom{\omega+1}{2}$ as a $\chi$-binding function.
Motivated by this, A. Gy{\'a}rf{\'a}s in \cite{gyarfas1987problems}, posed a problem which asks for the order of magnitude of the smallest $\chi$-binding function for $\mathcal{G}(2K_2)$ and recently, the best known upper bound for this class was given by M. Gei{\ss}er in \cite{maximillian2022}, as $\binom{\omega+1}{2}$- {\small{2}}$\lfloor\frac{\omega}{3}\rfloor$.
C. Brause et al. in \cite{brause2019chromatic}, proved that the class of $\{2K_2,H\}$-free graphs, where $H$ is a graph with $\alpha(H)\geq 3$, does not admit a linear $\chi$-binding function.
This raises a natural question of characterizing classes of $2K_2$-free graphs which admit linear $\chi$-binding functions.
In this direction, T. Karthick and S. Mishra in \cite{karthick2018chromatic}, proved that the families of $\{2K_2, H\}$-free graphs, where $H\in \{HVN, gem, K_1 + C_4, \overline{P_3 \cup P_2}, K_5-e, K_5\}$ admit special linear $\chi$-binding functions.
Some of these bounds were later improved by C. Brause et al. and A. Prashant et al. in \cite{brause2019chromatic} and \cite{dampk2} respectively.
There have been extensive studies to find linear subfamilies even for various superclasses of $2K_2$-free graphs.
See for instance \cite{maximillian2022,brause2019chromatic,prashantcaldam2023,bharathi2018colouring,caldam2022chromaticspringerversion,maximillian2022,schiermeyer2019polynomial}.
\begin{figure}
  \centering
  \includegraphics[width=7cm]{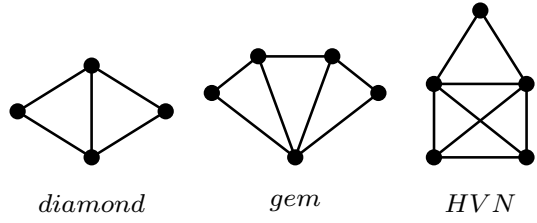}
  \caption{Some special graphs}\label{freegraphs}
\end{figure}
One of the superclass of $2K_2$-free graphs studied is the class of $(P_3\cup P_2)$-free graphs.
In \cite{bharathi2018colouring}, A. P. Bharathi and S. A. Choudum  gave a $\chi$-binding function of $O(\omega^3)$ for the class of $(P_3\cup P_2)$-free graphs and obtained a special linear $\chi$-binding function for $\{P_3\cup P_2,diamond\}$-free graphs.
A. Prashant et al. in \cite{caldam2022chromaticspringerversion}, have shown that $\{P_3\cup P_2, 2K_1+K_p\}$-free graphs and $\{P_3\cup P_2, (K_1\cup K_2)+K_p\}$-free graphs admit a linear $\chi$-binding function.
In \cite{choudum2010maximal}, S. A. Choudum et al. gave a decomposition theorem for $\{P_3\cup P_2,C_4\}$-free graphs and as a result showed that $f^*=\left\lceil\frac{5\omega}{4}\right\rceil$.
In \cite{wang2022chi}, X. Wang et al. showed that the class of $\{P_3\cup P_2, K_1+C_4\}$- free graphs admits $3\omega$ as a $\chi$-binding function.
Recently, T. Karthick and A. Char in \cite {char2022optimal}, showed that $\max\{\omega+3,\lfloor\frac{3\omega}{2}\rfloor-1\}$ is the optimal $\chi$-binding function for the class of $\{P_3\cup P_2,\overline{P_3\cup P_2}\}$-free graphs.

Moving in this direction, in this paper we show that the class of $\{P_3\cup P_2, gem\}$-free graphs is $\chi$-bounded by $2\omega$.
In addition, we give an example to establish that, unlike its subclass $\{2K_2, gem\}$-free graphs, which admits a special linear $\chi$-binding function (as shown by C. Brause in \cite{brause2019chromatic}),  the optimal $\chi$-binding function $f^*$ for $\{P_3\cup P_2, gem\}$-free graphs is at least  $\left\lceil\frac{5\omega}{4}\right\rceil$.
We also give examples to show that the $\chi$-binding function $f=2\omega$ is tight when $\omega=2 \textnormal{ and }3$.

Some of the graphs that are considered as forbidden induced subgraphs in this paper are given in Figure \ref{freegraphs}.
Notations and terminologies not mentioned here are as in \cite{west2005introduction}.

\section{Preliminaries}

Throughout this paper, we use a particular partition of the vertex set of a graph $G$ which was initially defined by S. Wagon in \cite{wagon1980bound} and later improved by A. P. Bharathi and S. A. Choudum  in \cite{bharathi2018colouring} as follows.
Let $A=\{v_1,v_2,\ldots,v_{\omega}\}$ be a maximum clique of $G$.
The lexicographic ordering on the set $L=\{(i, j): 1 \leq i < j \leq \omega\}$ is defined in the following way.
For two distinct elements $(i_1,j_1),(i_2,j_2)\in L$, we say that $(i_1,j_1)$ precedes $(i_2,j_2)$, denoted by $(i_1,j_1)<_L(i_2,j_2)$ if either $i_1<i_2$ or $i_1=i_2$ and $j_1<j_2$.
For every $(i,j)\in L$, let $C_{i,j}=\{v\in V(G)\backslash A:v\notin N(v_i)$ and \linebreak $\ v\notin N(v_j)\}\backslash\{\mathop\cup\limits_{(i',j')<_L(i,j)} C_{i',j'}\}$.
That is, for $(i,j)\in L$, $C_{i,j}$ consists of all the vertices $v$ in $V(G)\backslash A$ such that $i$ and $j$ are the least distinct positive integers for which $vv_i$ and $vv_j$ are not adjacent in $G$.
Clearly $C_{i,j}\cap C_{i',j'}=\emptyset$ whenever $(i,j)\neq (i',j')$.

For $1\leq k\leq \omega$, let $I_k=\{v\in V(G)\backslash A: v\in N(v_i), \text{\ for\ every}\ i\in\{1,2,\ldots,\omega\}\backslash\{k\}\} $.
Since $A$ is a maximum clique, for $1\leq k\leq \omega$, $I_k$ is an independent set and for any $x\in I_k$, $xv_k\notin E(G)$.
Also for any $(i,j)\in L$ and $k\in\{1,2,\ldots,\omega\}$, from the definitions of $C_{i,j}$ and $I_k$, we see that $I_k\cap C_{i,j}=\emptyset$.
Clearly, each vertex in $V(G)\backslash A$ is non-adjacent to at least one vertex in $A$. 
Hence those vertices will be contained either in $I_k$ for some $k\in\{1,2,\ldots,\omega\}$, or in $C_{i,j}$ for some $(i,j)\in L$.
Thus $V(G)=A\cup(\mathop\cup\limits_{k=1}^{\omega}I_k)\cup(\mathop\cup\limits_{(i,j)\in L}C_{i,j})$ and throughout the paper, we shall use this partition for $V(G)$.
Furthermore, we define two new sets $C'_{i,j}$ and $D_{i,j}$, where $C_{i,j}'$ denote the maximal subset of $C_{i,j}$ such that $\langle C'_{i,j}\rangle$ contains no isolated vertices and $D_{i,j}= A\backslash N_A(C_{i,j}')$.

Without much difficulty, one can make the following observations on $\left(P_3\cup P_2\right)$-free graphs.
\begin{fact}\label{true} Let $G$ be a  $\left(P_3\cup P_2\right)$-free graph. For $(i,j)\in L$,
the following holds.
\begin{enumerate}[(i)]
\setlength\itemsep{-1pt}
\item \label{CijClique} $\langle C_{i,j}\rangle$ is $P_3$-free and hence a disjoint union of cliques and $\chi(\langle C_{i,j}\rangle)=\omega(\langle C_{i,j}\rangle)$.
\item \label{joinvertex} If $a\in C_{i,j}$, then $N_A(a)\supseteq \{v_1,v_2,\ldots,v_j\}\backslash\{v_i,v_j\}$.
\end{enumerate}
\end{fact}
\begin{proof}
\begin{enumerate}[(i)]\setlength\itemsep{-1pt}
\item  Suppose a $\langle C_{i,j}\rangle$ contains a $P_3$ say $P$, then $\langle \{V(P),v_i,v_j\}\rangle\cong P_3\cup P_2$, contradiction.
\item Follows immediately from the definitions of $C_{i,j}$ and  $A$.
\vskip-.6cm
\end{enumerate}
\end{proof}

\section{$\{P_3\cup P_2, gem\}$-free graphs}

Let us start Section 3 by establishing $(3\omega-2)$ as a linear $\chi$-binding function for $\{P_3\cup P_2, gem\}$-free graphs.

\begin{proposition}\label{gembound}
Let $G$ be a $\{P_3\cup P_2, gem\}$-free graph, then $\chi(G)\leq 3\omega-2$.
\end{proposition}
\begin{proof}
Let $G$ be a $\{P_3\cup P_2, gem\}$-free graph.
It is easy to see that $[v_1,\{\mathop\cup\limits_{k=2}^{\omega}(v_i\cup I_k)\cup(\mathop\cup\limits_{i\geq 2}C_{i,j})\}]$ and $[v_2,\{(\mathop\cup\limits_{j\geq 3}C_{1,j})\cup v_1\cup I_1\}]$ are complete and $\omega(\langle\{\mathop\cup\limits_{k=2}^{\omega}(v_i\cup I_k)\cup(\mathop\cup\limits_{i\geq 2}C_{i,j})\}\rangle)=\omega-1$.
As a consequence, $\omega(\langle\{(\mathop\cup\limits_{j\geq 3}C_{1,j})\cup v_1\cup I_1\}\rangle)\leq \omega-1$ and since $G$ is $gem$-free, $\langle\{\mathop\cup\limits_{k=2}^{\omega}(v_i\cup I_k)\cup(\mathop\cup\limits_{i\geq 2}C_{i,j})\}\rangle$ and $\langle\{(\mathop\cup\limits_{j\geq 3}C_{1,j})\cup v_1\cup I_1\}\rangle$ are $P_4$-free and hence can be colored with $(\omega-1)$ colors each.
Furthermore, by (\ref{CijClique}) of Fact \ref{true}, $\chi(\langle C_{1,2}\rangle)\leq \omega(G)$.
Thus $\chi(G)\leq \chi(\langle V(G)\backslash C_{1,2}\rangle)+\chi(\langle C_{1,2}\rangle)\leq 2(\omega-1)+\omega=3\omega-2$ colors.
\end{proof}
Although without much difficulty, we have shown that the class of $\{P_3\cup P_2,gem\}$-free graphs is linearly $\chi$-bounded by $f=3\omega-2$, we shall show that this bound is not optimal.
Moving in this direction, we begin by recalling some properties about $gem$-free graphs in \cite{dampk2}.
 \begin{lemma}\cite{dampk2}\label{gemlemma1}
Let $G$ be a $gem$-free graph and $V(G)=A\cup(\mathop\cup\limits_{i=1}^{\omega}I_i)\cup(\mathop\cup\limits_{(i,j)\in L}C_{i,j})$.
For $i,j\in\{1,2,\ldots,\omega(G)\}$ such that $i<j$ and $j\geq 3$ the following holds.

\begin{enumerate}[(i)]

\item\label{Gperfect1} $\langle C_{i,j}\rangle$ is $P_4$-free and hence perfect.

\item\label{Gadjacency1} For any $l\in\{1,2,\ldots,\omega\}$, if $H$ is a component in $\langle C_{i,j}\rangle$ and $a\in V(H)$ such that $av_l\in E(G)$, then $[V(H),v_l]$ is complete.

\item\label{GIp} For $a\in C_{i,j}$ if $av_l\notin E(G)$, for some $l\in\{1,2,\ldots,\omega\}$, then $[a,I_l]=\emptyset$.

\item\label{Gclique1} If $H$ is a component of $C_{i,j}$, then $\omega(H)\leq |A\backslash N_A (V(H))|$.
\end{enumerate}
\end{lemma}

Using the properties given in Lemma \ref{gemlemma1}, we observe the following properties about $\{P_3\cup P_2,gem\}$-free graphs.

\begin{lemma}\label{gemlemma}
Let $G$ be a $\{P_3\cup P_2, gem\}$-free graph with $\omega(G)\geq 3$.
Let $i,j,k,\ell\in\{1,2,\ldots,\omega(G)\}$ and $j\geq 3$.
If $a\in C_{i,j}$, then the following are true.
\begin{enumerate}[(i)]
\setlength\itemsep{-1pt}

\item\label{Gadjacency} If $av_{\ell}\in E(G)$, then $[C_{i,j}',v_{\ell}]$ is complete. As a consequence, if $b\in C_{i,j}'$ and $bv_{\ell}\notin E(G)$, then $[C_{i,j},v_{\ell}]=\emptyset$.
	

\item\label{Gclique} $\omega(\langle C_{i,j}'\rangle)\leq |A\backslash N_A (C_{i,j}'))|$.

\item\label{cijcil} For $j,\ell\geq 3$, $[C_{i,j},C_{i,\ell}]=[C_{i,j},C_{j,\ell}]=\emptyset$. Furthermore, if $j\geq 4$, then $[C_{i,j}, C_{k,j}]=\emptyset$, for every $k\neq i$.

\item\label{edgerow} If $C'_{i,j}\neq\emptyset$, then for $j<\ell$, $C_{i,\ell}=C_{j,\ell}=\emptyset$ and for $\ell<j$, $C'_{i,\ell}=\emptyset$. Furthermore, if $j\geq 4$, then $C_{k,j}=\emptyset$, for every $k\neq i$.

\end{enumerate}
\end{lemma}

\begin{proof}
Let $G$ be a $\{P_3\cup P_2, gem\}$-free graph with $\omega(G)\geq 3$.
\begin{enumerate}[(i)]
\setlength\itemsep{-1pt}


\item If $av_{\ell}\in E(G)$, then clearly $\ell\neq\{i,j\}$.
If $\ell<j$, then by (\ref{joinvertex}) of Fact \ref{true},  $[C_{i,j},v_{\ell}]$ is complete and we are done.
Hence, let us assume $\ell>j$.
Suppose $[C'_{i,j},v_{\ell}]$ is not complete, then there exists a vertex $b\in C_{i,j}'$ such that $bv_{\ell}\notin E(G)$.
By (\ref{Gadjacency1}) of Lemma \ref{gemlemma1} and definition of $C_{i,j}'$, there exists another vertex $c\in C_{i,j}'$ such that $bc\in E(G)$ and $ab, ac, cv_{\ell}\notin E(G)$.
Hence $\langle \{a,v_{\ell},v_j,b,c\}\rangle\cong P_3\cup P_2$, a contradiction.
Thus $[C'_{i,j},v_{\ell}]$ is complete.



\item Follows immediately from (\ref{Gclique1}) of Lemma \ref{gemlemma1} and (\ref{Gadjacency}) of Lemma \ref{gemlemma}.

\item Let $a\in C_{i,j}$ and $b\in C_{i,\ell}\cup C_{j,\ell}$ such that $ab\in E(G)$.
Without loss of generality, let us assume $j<\ell$, then either $bv_j\ \text{or}\ bv_i$ is an edge in $G$. By (\ref{joinvertex}) of Fact \ref{true}, there exists a vertex $v_s\in A$ such that $\langle\{v_s, a,b,v_j,v_i\}\rangle\cong gem$, a contradiction.
Hence $[C_{i,j},C_{i,\ell}]=[C_{i,j},C_{j,\ell}]=\emptyset$.
Similarly, one can show that if $j\geq 4$, then for every $k\neq i$, $[C_{i,j}, C_{k,j}]=\emptyset$.

\item Let $C_{i,j}'\neq \emptyset$ and $a\in C'_{i,j}$.
Then there exists $b\in C_{i,j}'$ such that $ab\in E(G)$. 
Suppose for some $j<\ell$, there exists a vertex $c\in C_{i,\ell}\cup C_{j,\ell}$ then by (\ref{cijcil}), we have $\langle\{c,v_j,v_i,a,b\}\rangle\cong P_3\cup P_2$, a contradiction.
Hence $C_{i,\ell}=C_{j,\ell}=\emptyset$.
For $\ell< j$, if $C'_{i,\ell}\neq\emptyset$ then by using similar arguments, we can show that $C_{i,j}=\emptyset$, a contradiction.
Similarly, one can show that if $j\geq 4$, then $C_{k,j}=\emptyset$, for every $k\neq i$.

\end{enumerate}
\vskip-.92cm
\end{proof}

Now, let us establish $f=2\omega$ as a $\chi$-binding function for  $\{P_3\cup P_2, gem\}$-free graphs which is a better $\chi$-bound than the one given in Theorem \ref{gembound}.
\begin{theorem}\label{p2p3gem}
If $G$ is a $\{P_3\cup P_2, gem\}$-free graph, then $\chi(G)\leq2\omega(G)$.
\end{theorem}
\begin{proof}
Let $G$ be a $\{P_3\cup P_2, gem\}$-free graph and $\{1,2,\ldots,2\omega\}$ be the set of colors.
For $1\leq i\leq \omega$, let us assign the color $i$ to the vertices of $(v_i\cup I_i)$.
Among the remaining vertices of $V(G)$, we first color the vertices in $(\mathop\cup\limits_{(i,j)\in L'}C'_{i,j})$, where $L'=L\backslash (1,2)$ followed by all the vertices in $\mathop\cup\limits_{(i,j)\in L'} (C_{i,j}\backslash C'_{i,j})$ and finally, we color the vertices in $C_{1,2}$.
Before proceeding further let us observe the following property.

\textbf{Claim 1}: For $j,\ell,r\geq 3$, if $C_{r,s}\neq \emptyset$ then $[\{C'_{1,j}\cup C'_{2,\ell}\},\{v_r,v_s,C_{r,s}\}]$ is complete.

Let $a\in C_{r,s}$.
If $C'_{1,j}\cup C'_{2,\ell}=\emptyset$, then there is nothing to prove.
Hence without loss of generality, let us assume $C'_{1,j}\neq \emptyset$ and $b,c\in C'_{1,j}$ such that $bc\in E(G)$.
We begin by showing that $[C'_{1,j},\{v_r,v_s\}]$ is complete.
On the contrary, let us assume that $bv_r\notin E(G)$.
By (\ref{Gadjacency}) of Lemma \ref{gemlemma}, $[C'_{1,j},v_r]=\emptyset$.
If $ab\in E(G)$, then by (\ref{joinvertex}) of Fact \ref{true}, $\langle\{v_2,b,a,v_1,v_r\}\rangle\cong gem$, a contradiction.
Hence $ab\notin E(G)$ and for the same reason we see that $[a,C'_{1,j}]=\emptyset$ and thus $\langle\{a,v_1,v_r,b,c\}\rangle\cong P_3\cup P_2$, again a contradiction.
Hence $[C'_{1,j},v_r]$ is complete.
Similarly, one can show that $[\{C'_{1.j}\cup C'_{2,\ell}\},\{v_r,v_s\}]$ is complete.
Next, we shall show that $[\{C'_{1.j}\cup C'_{2,\ell}\}, C_{r,s}]$ is complete.
Suppose there exists a vertex $b\in C'_{1,j}$ such that $ab\notin E(G)$, then by (\ref{joinvertex}) of Fact \ref{true} and the above argument, $\langle\{v_2,a,v_1,v_r,b\}\rangle\cong gem$, a contradiction.
Hence $[C'_{1,j}, C_{r,s}]$ is complete and similarly, one can show that $[C'_{2,\ell},C_{r,s}]$ is complete.

To color the vertices in $(\mathop\cup\limits_{(i,j)\in L} C_{i,j})$, we begin by considering $\omega(G)\leq 2$.
One can see that if $\omega(G)\leq 2$, then $(\mathop\cup\limits_{(i,j)\in L} C_{i,j})\subseteq C_{1,2}$ and hence by (\ref{CijClique}) of Fact \ref{true}, the vertices of $C_{1,2}$ can be colored with colors $\{\omega+1,\ldots, 2\omega\}$ and hence $G$ will be $2\omega$-colorable.
Next, let us consider $\omega(G)\geq 3$.
To color the vertices of $(\mathop\cup\limits_{(i,j)\in L'} C_{i,j})$, we break our proof into two cases.

\noindent \textbf{Case 1}. $C_{r,s}$ is non-empty, for some $r\geq 3$.

By (\ref{GIp}) of Lemma \ref{gemlemma1} and (\ref{Gclique}) of Lemma \ref{gemlemma}, one can see that we can properly color each component of $\langle C'_{i,j}\rangle$, $j\geq3$, by using the colors given to the vertices in $D_{i,j}$.
We shall show that this coloring is proper.
Suppose for $j,\ell\geq 3$, there exist two adjacent vertices $a\in C'_{i,j}$ and $b\in C'_{k,\ell}$ such that they received the same color, say $q$ then $av_q, bv_q\notin E(G)$.
Without loss of generality, let us assume $(i,j)<_L(k,\ell)$.
By (\ref{edgerow}) of Lemma \ref{gemlemma}, $i<k$ and thus $bv_i\in E(G)$.
Let $t=\{1,2\}\backslash\{i\}$. If $k\geq 3$, then $\langle\{v_t,a,b,v_i,v_q\}\rangle\cong gem$, a contradiction.
If $k<3$, then the only possible values of $i$ and $k$ are $1$ and $2$ respectively.
Thus by Claim 1, $\langle\{v_r,a,b,v_1,v_q\}\rangle\cong gem$, again a contradiction.
Hence the coloring mentioned above is proper.

Next, let us color the vertices in $(\mathop\cup\limits_{(i,j)\in L'} C_{i,j}\backslash C'_{i,j})$.
For every $(i,j)\in L'$, let us assign the color $i$ to all the vertices in $(C_{i,j}\backslash C_{i,j}')$.
We claim that this extension of coloring is also proper.
Suppose the coloring is improper, then there exist a pair of adjacent vertices $a$ and $b$ which receive the same color $i$, where $a\in C_{i,j}\backslash C_{i,j}'$ and $b\in C_{k,\ell}$ for some $j,\ell\geq 3$.
By (\ref{cijcil}) and (\ref{edgerow}) of Lemma \ref{gemlemma} and our coloring technique, one can observe that $k<i$, $b\in C_{k,l}'$ and $bv_i\notin E(G)$.
Next, we shall show that $i>2$.
Suppose $i\leq2$, then $k=1$ and $i=2$ is the only possibility and since $\ell\geq 3$, $[C_{1,\ell},v_2]$ is complete and hence $bv_2\in E(G)$, a contradiction.
Thus $i>2$.
Let $s=\{1,2\}\backslash\{k\}$. 
Since $k<i$, $bv_i\notin E(G)$ and $i\geq 3$, $\langle\{v_s,b,a,v_k,v_i\}\rangle\cong gem$, a contradiction.
Hence if $C_{r,s}\neq \emptyset$, for some $r\geq 3$, then the vertices in $V(G)\backslash C_{1,2}$ can be properly colored with $\{1,2,\ldots,\omega\}$ colors.

\noindent \textbf{Case 2}. $C_{r,s}=\emptyset$, for every $r\geq 3$.

Since $C_{r,s}=\emptyset$, for every $r\geq 3$, $(\mathop\cup\limits_{(i,j)\in L} C_{i,j})\backslash C_{1,2}= (\mathop\cup\limits_{p=3}^{\omega}C_{1,p})\cup (\mathop\cup\limits_{q=3}^{\omega}C_{2,q})$.
We begin by coloring the vertices in $(\mathop\cup\limits_{p=3}^{\omega}C'_{1,p})\cup (\mathop\cup\limits_{q=3}^{\omega}C'_{2,q})$.
By (\ref{edgerow}) of Lemma \ref{gemlemma}, there exist unique $j,\ell\geq 3$, such that $\mathop\cup\limits_{p=3}^{\omega}C'_{1,p}= C'_{1,j}$ and $\mathop\cup\limits_{q=3}^{\omega}C'_{2,q}= C'_{2,\ell}$.
If $C_{1,j}'\cup C_{2,\ell}'=\emptyset$, then there is nothing to prove.
If $C'_{1,j}\neq \emptyset$ and $C'_{2,\ell}=\emptyset$, then by (\ref{GIp}) of Lemma \ref{gemlemma1} and (\ref{Gadjacency}) and (\ref{Gclique}) of Lemma \ref{gemlemma}, we can assign the colors given to the vertices of $D_{1,j}$ to the vertices of $C'_{1,j}$.
We can use a similar technique when $C'_{1,j}=\emptyset$ and $C'_{2,\ell}\neq \emptyset$.
Next, let us assume that both $C'_{1,j}$ and $C'_{2,\ell}$ are non-empty.
If ${D_{1,j}\cap D_{2,\ell}}=\emptyset$, we can do the same way by assigning the colors given to the vertices of $D_{1,j}$ and $D_{2,\ell}$ to the vertices of $C_{1,j}'$ and $C_{2,\ell}'$ respectively.
Hence let us assume that ${D_{1,j}\cap D_{2,\ell}}\neq\emptyset$.
In this case, we begin by showing that $N_A (C_{1,j}') \cap N_A (C_{2,\ell}')= \emptyset$.
Suppose there exist vertices $a\in C_{1,j}'$, $b\in C_{2,\ell}'$ such that  $N_A (a)\cap N_A (b)\neq \emptyset$.
Let $v_x\in N_A (a) \cap N_A (b)$ and $v_y\in {D_{1,j}\cap D_{2,\ell}}$.
If $ab\notin E(G)$, then $\langle \{v_x,a,v_2,v_1,b\}\rangle\cong gem$, a contradiction and if $ab\in E(G)$, then $\langle \{v_x,a,b,v_1,v_y\}\rangle\cong gem$, again a contradiction.
Hence if ${D_{1,j}\cap D_{2,\ell}}\neq \emptyset$, then $N_A (C_{1,j}') \cap N_A (C_{2,\ell}')= \emptyset$.
Let us further divide the proof into two cases based on $|{D_{1,j}\cap D_{2,\ell}}|$.

\noindent \textbf{Case 2.1}. $|{D_{1,j}\cap D_{2,\ell}}|\geq 2$.

Let $v_x,v_y\in D_{1,j}\cap D_{2,\ell}$.
We begin by showing that $[C_{1,j}',C_{2,\ell}']$ is complete.
Suppose there exist vertices $a\in C_{1,j}'$ and $b\in C_{2,\ell}'$ such that $ab\notin E(G)$.
Let $c$ be an another vertex in $C_{1,j}'$ such that $ac\in E(G)$.
If $bc\notin E(G)$, $\langle \{b,v_1,v_x,a,c\}\rangle\cong P_3\cup P_2$, a contradiction and if $bc\in E(G)$, $\langle \{a,c,b,v_x,v_y\}\rangle\cong P_3\cup P_2$, again a contradiction.
Hence $[C_{1,j}',C_{2,\ell}']$ is complete.
By (\ref{CijClique}) of Fact \ref{true}, we know that each $\langle C_{i,j}\rangle$ is a union of cliques.
Let $S\subseteq C_{1,j}'$ and $T\subseteq C_{2,\ell}'$ such that $|S|=\omega(\langle C_{1,j}'\rangle)$ and $|T|=\omega(\langle C_{2,\ell}'\rangle)$ respectively.
Clearly, $|S\cup T|\leq \omega(G)=|A|$.
Since $N_A (C_{1,j}') \cap N_A (C_{2,\ell}')= \emptyset$, we have $|N_A (S)|+| N_A (T)|+|(D_{1,j}\cap D_{2,\ell})|=|A|\geq |S\cup T|$ and by (\ref{Gclique}) of Lemma \ref{gemlemma}, we have $|S|\leq  |D_{1,j}|=| N_A (T)|+|(D_{1,j}\cap D_{2,\ell})|$ and $|T|\leq  |D_{2,\ell}|=|N_A (S)|+|(D_{1,j}\cap D_{2,\ell})|$.
Let us start coloring the vertices of $S$ and $T$ with the colors given to the vertices of $|N_A(T)|$ and $|N_A(S)|$  respectively.
If there are uncolored vertices in $S\cup T$, then color those vertices with the colors given to the vertices in $D_{1,j}\cap D_{2,\ell}$.
By (\ref{GIp}) of Lemma \ref{gemlemma1} and the above inequalities, it is easy to observe that the given coloring is proper.
By (\ref{GIp}) of Lemma \ref{gemlemma1} and (\ref{Gadjacency}) of Lemma \ref{gemlemma}, the remaining components in $\langle C_{1,j}'\rangle$ and $\langle C_{2,\ell}'\rangle$ can be colored  by using the colors assigned to the vertices of $S$ and $T$ respectively.
Clearly, this is a proper coloring.

\noindent \textbf{Case 2.2}. $|{D_{1,j}\cap D_{2,\ell}}|=1$.

In this case we assign the colors given to the vertices of $D_{2,\ell}$ and $D_{1,j}\backslash D_{2,\ell}$ together with a new color $\omega+1$ to the vertices of $C_{2,\ell}'$ and $C_{1,j}'$ respectively.

By using (\ref{GIp}) of Lemma \ref{gemlemma1} and (\ref{cijcil}) of Lemma \ref{gemlemma} and (\ref{joinvertex}) of Fact \ref{true}, we can assign the colors $1$ and $2$ to the vertices in $(\mathop\cup\limits_{p=3}^{\omega}C_{1,p})\backslash C_{1,j}'$ and $(\mathop\cup\limits_{q=3}^{\omega}C_{2,q})\backslash C_{2,\ell}'$ respectively.
Clearly, this coloring is proper.

Finally, let us color the vertices of $C_{1,2}$.
If $\omega(\langle C_{1,2}\rangle)\leq\omega(G)-1$, then by (\ref{CijClique}) of Fact \ref{true}, we can color the vertices of $C_{1,2}$ with colors $\{\omega+2,\ldots,2\omega\}$ and thus $\chi(G)\leq2\omega(G)$.
Hence let us consider the case when $\omega(\langle C_{1,2}\rangle)=\omega(G)$.
Let $S\subseteq C_{1,2}$ such that $\omega(\langle S\rangle)=\omega(\langle C_{1,2}\rangle)=\omega(G)$.
If none of the vertices in $ C_{1,j}'$ have been assigned the color $\omega(G)+1$, then the vertices of $C_{1,2}$ can be colored with colors $\{\omega+1,\ldots,2\omega\}$ and hence $G$ is $2\omega$-colorable.
Next, let us assume that there are components in $\langle C_{1,j}'\rangle$ that have been assigned the color $\omega(G)+1$.
Say $H_1,\ldots,H_t$, $t\geq1$.
Clearly, for a vertex $a\in C_{1,j}$, $|[a,S]|\leq |S|-1=\omega(G)-1$.
Also, it is not difficult to see that $|[a,S]|\geq |S|-1$, otherwise there would exist vertices $x,y\in S$ such that $ax,ay\notin E(G)$ and $\langle \{a,v_2,v_1,x,y\}\rangle\cong P_3\cup P_2$, a contradiction.
Thus $|[a,S]|= |S|-1=\omega(G)-1$.
Let $u_1,\ldots,u_t$ be vertices in components $H_1,\ldots,H_t$ respectively that have been assigned the color $\omega(G)+1$.
We claim that there exists a vertex $z\in S$ such that $[z,\{u_1,\ldots,u_t\}]=\emptyset$.
The claim holds when $t=1$.
For $t\geq2$ and any $z\in S$, suppose that $[z,\{u_1,\ldots,u_t\}]\neq\emptyset$, then there exist vertices $u_a\in V(H_a),u_b\in V(H_b),x,y\in S$ such that $x\neq y$ and $u_ax,u_by\notin E(G)$.
Since $\omega(G)\geq3$, there exist a vertex $w\in S\backslash\{x,y\}$ such that $wu_a,wu_b\in E(G)$ (since, the vertices $u_a,u_b$ can be non-adjacent to exactly one vertex each in $S$) and thereby $\langle \{w,u_a,y,x,u_b\}\rangle\cong K_1+ P_4$, a contradiction.
Thus there exists a vertex $z\in S$ such that $[z,\{u_1,\ldots,u_t\}]=\emptyset$.
Now, color the vertex $z$ with the color $\omega(G)+1$ and the remaining vertices of $S$ with colors $\{\omega(G)+2,\ldots, 2\omega(G)\}$.
Similarly, we can color all the other components having size $\omega(G)$ in $\langle C_{1,2}\rangle$ and the remaining components can be properly colored with colors $\{\omega+2,\omega+3,\ldots, 2\omega\}$.
Hence $\chi(G)\leq2\omega(G)$.
\end{proof}
\section{Conclusion}

In this paper we have shown that every $\{P_3\cup P_2,gem\}$-free graph is $2\omega$-colorable.
A natural question is whether this bound is tight.
With the help of some examples we shall show that the $\chi$-binding function $f=2\omega$ is tight for $\omega=2\textnormal{ and }3$, the class of $\{P_3\cup P_2,gem\}$-free graphs cannot admit a special linear $\chi$-binding function and that $f^*\geq \left\lceil\frac{5\omega(G)}{4}\right\rceil$.
However, the question of whether the bound is tight for $\omega\geq 4$ is still open.

In \cite{bharathi2018colouring}, it has already been established that the Gr\"{o}tzsch graph $\mu(C_5)$ is an example of a $\{P_3\cup P_2, diamond\}$-free graph with $\omega(\mu(C_5))=2$ and $\chi(\mu(C_5))=4$.
In addition, it was observed in \cite{skarthick2018chromatic}, that the complement of Schl\"{a}fli graph $G$ (see, \textcolor{blue}{www.distanceregular.org/graphs/complement-schlafli.html}) is an example of a $\{P_3\cup P_2, diamond\}$-free graph with $\omega(G)=3$ and $\chi(G)=6$.
Since every $\{P_3\cup P_2, diamond\}$-free graph is also $\{P_3\cup P_2, gem\}$-free, the bound given in Theorem \ref{p2p3gem} is optimal for $\omega=2\textnormal{ and }3$ 

Next, let us recall the  definition of complete expansion of a graph $G$.
Let $G$ be a graph on $n$-vertices $\{v_1,v_2,\ldots,v_n\}$ and let $H_1,H_2,\ldots,H_n$ be $n$ vertex-disjoint graphs.
An expansion $G(H_1,H_2,\ldots,H_n)$ of $G$ is a graph obtained from $G$ by 

\begin{enumerate}[(i)]
\setlength\itemsep{-1pt}
\item replacing each $v_i$ of $G$ by $H_i$, $i=1,2,\ldots,n$ and
\item by joining every vertex in $H_i$ with every vertex in $H_j$, whenever $v_i$ and $v_j$ are adjacent in $G$.
\end{enumerate}

For $i\in \{1,2,\ldots,n\}$, if $H_i\cong K_{m_i}$, then $G(H_1,H_2,\ldots,H_n)$ is said to be a complete expansion of $G$ and is denoted by $\mathbb{K}[G](m,m,m,m,m)$ or $\mathbb{K}[G]$.
Let us consider the graph $G\cong \mathbb{K}[C_5](m,m,m,m,m)$, where $\omega(G)=2m$, $m\geq 1$.
Clearly, $\alpha(G)=2$ and thus $G$ is $(P_3\cup P_2)$-free and any path on four vertices cannot have a common neighbor, thus $G$ is $gem$-free. In \cite{fouquet1995graphs}, it was shown that $\chi(G)=\left\lceil\frac{5\omega(G)}{4}\right\rceil$. Hence the class of $\{P_3\cup P_2,gem\}$-free graphs cannot have a special linear $\chi$-binding function.

\bibliographystyle{ams}
\bibliography{ref}
\end{titlepage}
\end{document}